\documentclass[11pt]{amsart}
\pdfoutput=1 
\usepackage{todonotes} 
\usepackage[pagebackref,pdfauthor={Yiannis N. Petridis, Morten S. Risager},
pdftitle={},
            pdfkeywords={Hyperbolic lattice points, Quantum unique ergodicity},
             pdfcreator={Pdflatex}]
{hyperref}
\hypersetup{colorlinks=true}
\usepackage{xcolor}
\usepackage{pdfpages}
\usepackage[normalem]{ulem}
\usepackage{marginnote}

\usepackage{amsfonts,amssymb,amsmath,amsthm}
\usepackage{url}
\usepackage{enumerate}

\newtheorem{theorem}{Theorem}[section]
\newtheorem{lem}[theorem]{Lemma}
\newtheorem{prop}[theorem]{Proposition}

\theoremstyle{definition}

\newtheorem{conj}[theorem]{Conjecture}
\newtheorem{remark}[theorem]{Remark}
\numberwithin{equation}{section}

\def \l {{\lambda}}

\def \e {{\varepsilon}}

\def \g {{\gamma}}
\def \G {{\Gamma}}
\def \R {{\mathbb R}}
\def \H {{\mathbb H}}
\def \C {{\mathbb C}}
\def \Z {{\mathbb Z}}

\def \GmodH {{\Gamma\backslash\H}}

\def \psl  {{\hbox{PSL}_2( {\mathbb R})} }
\def \pslz  {{\hbox{PSL}_2( {\mathbb Z})} }

\newcommand{\norm}[1]{\left\lVert #1 \right\rVert}
\newcommand{\abs}[1]{\left\lvert #1 \right\rvert}

\newcommand{\vol}[1]{\hbox{vol}( #1 )}

\title{Local average in hyperbolic lattice point counting}    
\author{Yiannis N. Petridis}
\address{Department of Mathematics, University College London, Gower Street, London WC1E 6BT, United Kingdom}
\email{i.petridis@ucl.ac.uk}

\author{Morten S. Risager}
\address{Department of Mathematical
  Sciences, University of Copenhagen, Universitetsparken 5, 2100
  Copenhagen \O, Denmark}
\email{risager@math.ku.dk}

\thanks{The second author was supported by a Sapere Aude grant from The Danish Council for Independent Research.}
\date{\today}

\makeatletter
\newcommand{\addresseshere}{%
  \enddoc@text\let\enddoc@text\relax
}
\makeatother

\makeatletter
\let\@wraptoccontribs\wraptoccontribs
\makeatother

\begin{document}
\contrib[with an appendix by]{Niko Laaksonen}
\begin{abstract}  The hyperbolic lattice point problem asks to
  estimate the size of the orbit $\Gamma z$ inside a hyperbolic disk
  of radius $\cosh^{-1}(X/2)$ for $\Gamma$ a discrete subgroup of
  $\psl$.  Selberg proved the  estimate $O(X^{2/3})$ for the error
  term for cofinite or cocompact groups. This has not been improved
  for any group and any center. In this paper  local averaging over
  the center is investigated for $\pslz$. The result is that the error
  term can be improved to $O(X^{7/12+\e})$.  The proof uses
  surprisingly strong input e.g. results on the quantum ergodicity of
  Maa{\ss} cusp forms and estimates on spectral exponential sums.  We
  also prove omega results for this averaging, consistent with the
  conjectural best error bound $O(X^{1/2+\e})$. In the appendix the
  relevant exponential sum over the spectral parameters is investigated.\end{abstract}
\maketitle
\section{Introduction}

Let $d$ be the hyperbolic distance on the upper half-plane $\H$, and 
$u$ the standard point-pair invariant
\begin{equation*}u(z, w)=\frac{\abs{z-w}^2}{4\Im z\Im w}.\end{equation*}
Let $\Gamma $ be a discrete subgroup of $\psl$, the group of isometries of $\H$.
Let $N(z, w, X)$ be defined
as \begin{equation}\label{standard-counting-function}
N(z, w, X)=\#\{\gamma\in \Gamma , 4u(\gamma z, w)+2\le X \}
\end{equation}
the condition being equivalent to $ d(\gamma z, w)\le
\cosh^{-1}(X/2)$, i.e. $N(z, w, X)$ counts the number of lattice
points $\gamma z$ within the hyperbolic circle of radius
$R=\cosh^{-1}(X/2)$ centered at $w$. Understanding this function is traditionally called the hyperbolic
lattice point problem. This problem has a long history,
\cite{Delsarte:1942a, Huber:1959a, Patterson:1975a, Gunther:1980a,
  PhillipsRudnick:1994a, Iwaniec:2002a}, and several generalizations
 \cite{BruggemanGrunewaldMiatello:2011, Kontorovich:2009, GorodnikNevo:2012}.
We restrict our attention to the case where $\Gamma$ is cocompact or cofinite. Unlike the Euclidean lattice
point problem, there is no known elementary or geometric way of
finding asymptotics for the
counting function, as the length and the area of a hyperbolic circle are of the same order of growth. 

The problem is related to the pre-trace formula or eigenfunction
expansion of an integral kernel for $\GmodH$. The main term in the
asymptotic expansion is
\begin{equation*}
M(z, w, X)=\sqrt{\pi}\sum_{s_j\in (1/2, 1]}\frac{\Gamma(s_j-1/2)}{\Gamma (s_j+1)}u_j(z)\overline{u_j(w)}X^{s_j},
\end{equation*}
where $\lambda_j=s_j(1-s_j) $ are the small eigenvalues of the automorphic Laplacian, i.e. are less than $1/4$. 
A central problem is to understand the growth of the error term
i.e. of $N(z,w,X)-M(z,w,X)$. The best known error bound is
\begin{equation}\label{selberg-bound}
  N(z,w,X)=M(z,w,X)+O(X^{2/3}),
\end{equation}
due to Selberg (1970's unpublished, see \cite[Theorem
12.1]{Iwaniec:2002a} and  also \cite{Good:1983b}). This bound holds for \emph{any} cofinite group, and no
group with better bound is known.  For congruence groups, like e.g. $\pslz$
the error term is conjectured to be of the order
$O(X^{1/2+\varepsilon})$. If true, this is essentially optimal possibly up to
changing $X^\varepsilon$ with powers of $\log\log X$ (See
\cite{PhillipsRudnick:1994a}). 

For the rest of the paper we restrict to $\Gamma=\pslz$. In this case
the main term simplifies to $M(z,w,X)=\frac{\pi}{\vol{\GmodH}}X$, as there are no small eigenvalues. We investigate \emph{local averages}  of the hyperbolic
lattice point counting, i.e. we vary the center of the hyperbolic circle locally. We get an improvement on the  exponent
$2/3$ on average.  To be precise we study the function
\begin{equation}
  \label{function}
N_f(X)= \int_\GmodH f(z)N(z,z,X)d\mu(z), 
\end{equation} where  $f$ is  a smooth, compactly supported function on $\GmodH$. 
We prove the following theorem.
\begin{theorem}\label{maintheorem} Let $\Gamma$ be ${ \pslz}$, and
  assume that $f$  is compactly supported on  $\GmodH$, smooth and nonnegative. Then
  \begin{equation*}
N_f(X) =\pi  X \overline f+O_{f, \e}(X^{7/12+\varepsilon}),
\end{equation*}
where $\displaystyle \overline f=\frac{1}{\vol\GmodH}\int_{\GmodH}f(z)d\mu(z)$.
\end{theorem}
It is tempting to speculate that also in the case of local averaging
the order of growth of the error term should be  $X^{1/2+\varepsilon}$, i.e. 
\begin{equation}
  \label{conjectural-error-term}
  N_f(X)=\pi X \overline f +O_{f, \e}(X^{1/2+\varepsilon}).
\end{equation}
We prove an omega result, which is consistent with (\ref{conjectural-error-term}):
\begin{theorem}\label{lowerbound}
   Let $\Gamma={ \pslz}$, and
  assume that $f$ is a nonzero, nonnegative, smooth, compactly supported
   function on
  $\GmodH$.  Then for every $\nu>0$ we have
  \begin{equation*}
    N_f(X) =\pi  X \overline f+\Omega(X^{1/2}(\log\log X)^{1/4-\nu}).
  \end{equation*}
\end{theorem}

\begin{remark}
  Theorem \ref{maintheorem} improves (on average) the error term \eqref{selberg-bound}, but only by
  going \emph{halfway} between Selberg's exponent $2/3$ and the expected one $1/2+\e$. Our proof
  is only valid for groups similar to  $\pslz$, as it requires surprisingly
  strong arithmetic input not available for most groups. Among other input  we use  effective error terms on average for  the 
  mass equidistribution of Maa{\ss} cusp forms (see Theorem
  \ref{average-best-qe} below), which itself follows from a remarkable
result of  Luo--Sarnak on the   mean 
  Lindel\"of hypothesis for Rankin--Selberg convolutions in the spectral aspect \cite{LuoSarnak:1995a}. 
\end{remark}

\begin{remark}
 Using a spectral
  large sieve, Chamizo
  \cite{Chamizo:1996a,Chamizo:1996b} proved that  the mean square over  the interval $[X, 2X]$ of  the error term
  is of the expected size, i.e he proved
  \begin{equation}\label{chamizo-integral}
    \left(\frac{1}{X}\int_X^{2X}\abs{N(z,w,X)-M(z,w,X)}^2dX\right)^{1/2}=O(X^{1/2}\log
    X).
  \end{equation}
In fact he proved a more precise statement, namely
 that if one fixes
 $z$, $w$ and takes sufficiently many, sufficiently  well-spaced  radii,
  then the second  moment of the absolute value of the error term
   has the average bound consistent with the optimal error term
   $O(X^{1/2+\varepsilon}$)  (up to $X^\varepsilon$ being replaced
   by powers of $\log X$).  The  $L^2$-estimate \eqref{chamizo-integral} follows easily from this.
\end{remark}

\begin{remark}Theorem \ref{lowerbound} is a local average analogue of
  the pointwise omega  result proved by Phillips and Rudnick
  \cite[Theorem 1.2]{PhillipsRudnick:1994a}. Our proof follows to a
  large extend that of  \cite[Theorem 1.2]{PhillipsRudnick:1994a},
  with some differences, due to the non-uniformity in $(z,w)$ of their
  result. We also make extensive use of known properties of certain special
  functions  simplifying some arguments in their proof. The essential idea of the proof goes
  back to Hardy \cite[p. 23--25]{Hardy:1917}.
\end{remark}
\begin{remark} There is another approach to the  hyperbolic lattice point counting for $\pslz$, 
due to Huxley and Zhigljavsky \cite{HuxleyZhigljavsky:2001a} using Farey fractions. They get an asymptotic formula for the number of pairs of consecutive  fractions in the Farey sequence subject to certain restrictions.  This approach has not been investigated further in comparison with  the application of the spectral theory of automorphic forms.
\end{remark}
\begin{remark}
 Hill and Parnovski in \cite{HillParnovski:2005a} studied the variance of $N(z, w, X)-{\pi}X/{\vol{\GmodH}}$ in the $w$ variable. To simplify their result, we assume that there are no eigenvalues $\lambda_j\le 1/4$, $\Gamma$ is cocompact and that we work in the two-dimensional hyperbolic space. Then
 $$\int_{\GmodH}\left|N(z, w, X)-\frac{\pi}{\vol{\GmodH}}X\right|^2\, d\mu(w)=O(X),$$
 see \cite[Eq.~(8)]{HillParnovski:2005a}.
\end{remark}
\begin{remark}
 We set $N(R)=\#\{(m, n)\in {\mathbb Z}^2, m^2+n^2\le R^2\}$ and $E(R)=N(R)-\pi R^2$. The  function $N(R)$ is counting the average number $r(n) $ of representations of an integer $n$ as sum of two squares.   The Gauss circle problem asks to estimate $E(R)$. Hardy's conjecture (unproved) states that
  $E(R)=O_{\e}(R^{1/2+\e})$, while the best upper bound known is
  $E(R)=O_{\e}(R^{131/208+\e})$, due to Huxley
  \cite{Huxley:2003a}. Hardy \cite{Hardy:1917} has proved the omega result $E(R)=\Omega (R^{1/2}\log^{1/4}R)$. Cram\'er \cite{Cramer:1918} provided mean-square asymptotics for $E(R)$:
  $$\int_1^R E(x)^2\, dx=c\cdot R^2+O(R^{3/2}), \quad c=\frac{1}{4\pi^2}\sum_{n=1}^{\infty}\frac{r(n)^2}{n^{3/2}}.$$
\end{remark}
\begin{remark}
The structure of the paper is as follows.
We discuss background material on exponential sums over the
eigenvalues, the rate of quantum ergodicity of eigenfunctions, the
pre-trace formula,  and approximations to the hyperbolic lattice-point
problem in Sections \ref{pnt}, \ref{QUE}, \ref{PRE}, \ref{approx}. The
proof of Theorem \ref{maintheorem} is in Section \ref{hedgehog} and
the proof of Theorem \ref{lowerbound} in Section
\ref{hedgehogbabies}. 

\end{remark}
\section{Prime geodesic theorems and exponential sums over the eigenvalues}\label{pnt}
The  hyperbolic lattice point problem and the prime geodesic theorem  on hyperbolic surfaces are two problems where the spectral theory of automorphic forms can be used, via the pre-trace resp. trace formula. To get good error terms in
 the prime geodesic theorem
Iwaniec \cite{Iwaniec:1984a} proved the following explicit formula for $\pslz$:
\begin{equation*}
\sum_{N(P)\le x}\log N(P_0)=x+\sum_{|t_j|\le T}\frac{x^{s_j}}{s_j}+O\left(\frac{x}{T}\log ^2x\right)
\end{equation*}
for $T\le x^{1/2}(\log x)^{-2}$. Here $P$ is a conjugacy class of
a hyperbolic element,  $P_0$ is the related primitive conjugacy class,
and $N(P_0)$ is its norm.  This shows that one cannot expect an error term better
than $x^{3/4+\e}$ without some cancellation
in the sum over eigenvalues, due to Weyl's law (\cite[Corollary 11.2]{Iwaniec:2002a}).
 Let us define 
\begin{equation}\label{exp-sum}S(T, X)=\sum_{|t_j|\le T}X^{it_j}.\end{equation}
Using Weyl's law,  we have the trivial estimate
\begin{equation}
  \label{trivial-bound}
  S(T,X)=O(T^2).
\end{equation}
Iwaniec \cite{Iwaniec:1984a}
proved that 
\begin{equation}\label{Iwaniec-bound}	S(T, X)=O(X^{{11/48}+\e}T),
\end{equation}
from which he deduced that
\begin{equation*}
\pi (x)=\hbox{li}(x)+O(x^{35/48+\e}),
\end{equation*}
where $\pi(x)=  \{P_0, N(P_0)\le x\}$.
Luo--Sarnak \cite[Theorem 1.2]{LuoSarnak:1995a} proved the following result.
\begin{theorem}\label{luosarnak-bound} For the exponential sum (\ref{exp-sum}) the following estimate holds:
\begin{equation*}
S(T,X)=O(X^{1/8}T^{5/4}(\log T)^2).
\end{equation*}
\end{theorem}
We notice that the exponent of $X$ is smaller than in
\eqref{Iwaniec-bound} while the exponent of $T$ is larger.
Theorem \ref{luosarnak-bound} allowed Luo and Sarnak to prove
\begin{equation*}
\pi (x)=\hbox{li}(x)+O(x^{7/10+\e}).
\end{equation*}
Very recently Soundararajan and Young \cite{SoundararajanYoung:2013} proved that
\begin{equation*}
\pi (x)=\hbox{li}(x)+O(x^{25/36+\e})
\end{equation*}
for $\pslz$ with an entirely different method.
One aim of this work is to show how to use cancellation in the  exponential sum $S(T, X)$
in the hyperbolic lattice point problem.
  We conjecture square root cancellation in (\ref{exp-sum}) with uniform dependence on $X$:
\begin{conj}\label{conjectural-bound} Let $X>2$. For the exponential sum $S(T, X)$ we have
\begin{equation*}
S(T,X)=O(X^{\e} T^{1+\e}).
\end{equation*}
\end{conj}
 This conjecture will give the best possible error term in the prime geodesic
 theorem up to powers of $\log X$. We will see that Conjecture \ref{conjectural-bound} implies also
 the best possible error term on average for the hyperbolic lattice
 point problem, i.e. \eqref{conjectural-error-term}. In fact we
 shall see in Remark \ref{use-remark}  that 
 \begin{equation}\label{sufficient}
S(T,X)=O(X^\e T^{3/2-\delta})
\end{equation}
for some $\delta>0$ suffices to prove \eqref{conjectural-error-term}.

In the Appendix, N. Laaksonen investigates the conjecture numerically and proves a theorem for the exponential sum $S(T, X)$ for fixed $X$, as $T\to\infty$. The numerics and the theorem point to the correctness of Conjecture~\ref{conjectural-bound}.

 \section{Quantum ergodicity}\label{QUE}
For $u_j$ an orthonormal basis of eigenfunctions we form the measures
\begin{equation}d\mu_j=\abs{u_j(z)}^2d\mu(z).\end{equation}
For $\pslz$ Lindenstrauss \cite{Lindenstrauss:2006a} and Soundararajan \cite{Soundararajan:2010a} have proved recently that for Hecke-Maa\ss{} eigenforms the Quantum Unique Ergodicity conjecture holds, i.e.
\begin{equation}\label{quantum-ergo}
d\mu_j\to \frac{1}{\vol{\GmodH}}d\mu (z), \quad j\to\infty.
\end{equation}
The question of the \emph{rate} of convergence of (\ref{quantum-ergo}) has
been raised by Sarnak \cite[Eq. 3.7]{Sarnak:1995a}, who conjectured that 
\begin{equation}\label{optimal-bound}
 \int_{\GmodH} f(z)d\mu_j(z)-\overline{ f} =O(t_j^{-1/2+\e}),
\end{equation}
where $\overline f ={\vol \GmodH}^{-1}\int_{\GmodH}f(z)d\mu(z)$.
 For general hyperbolic surfaces Zelditch \cite{Zelditch:1994} proved
 \begin{equation*}
 \sum_{|t_j|\le T}\left|   \int_{\GmodH} f(z)d\mu_j(z)-\overline{ f} \right|^{2k}=O\left(\frac{T^{2}}{\log^k T}\right).
 \end{equation*}
For $\pslz$ Luo--Sarnak \cite{LuoSarnak:1995a} proved  the optimal bound \eqref{optimal-bound} on average:
\begin{theorem}\label{average-best-qe} Let $f$ be a smooth compactly supported function on $\GmodH$.
\begin{equation*}
\sum_{|t_j|\le T}\left|   \int_{\GmodH} f(z)d\mu_j(z)-\overline{ f} \right|^2=O(\norm{f}_{8, 8}^2T^{1+\e}),
\end{equation*}
where the constant depends only on $\e$. 
\end{theorem}The norm $\norm{f}_{8, 8}$ is finite for $f$ smooth and compactly
supported. We use Theorem \ref{average-best-qe} in Lemma \ref{good} to get
rid of the eigenfunctions in the integrated pre-trace formula (see
Proposition \ref{local-trace-formula} below).

\section{Integrated pre-trace formula}\label{PRE}
Let $k\in \C^\infty(\R_+)$ be a function with Selberg--Harish-Chandra
transform $h(t)$ (see \cite[(1.62)]{Iwaniec:2002a}  for its definition)  even, holomorphic in $\abs{\Im t}\leq
1/2+\e$, and $h(t)=O(1/(1+\abs{t})^{2+\e})$ in the
strip.

Let
\begin{equation}
  \label{eq:1}
 K(z,w)=\sum_{\g\in\G}k(u(\gamma z,w))  
\end{equation}
be the corresponding automorphic kernel, and $K$ the corresponding
integral operator. 
\begin{prop}\label{local-trace-formula}
  Let $f$ be a smooth, compactly supported function on $\GmodH$.  Then we have
  \begin{align*}
    \label{eq:2}
    \int_{\GmodH}f(z)K(z,z)d\mu(z)=&\overline
    f
\sum_{t_j}h(t_j) 
+\sum_{t_j}h(t_j)\left(\int_\GmodH f(z)d\mu_j(z)-\overline
  f\right)\\ 
&\quad +\frac{1}{4\pi}\int_\R
h(t)
\int_{\GmodH}f(z)\abs{E(z,1/2+it)}^2d\mu(z)
dt
  \end{align*}
with absolute convergence on the right-hand side.
\end{prop}
\begin{proof}
  We set $z=w$ in Selberg's pre-trace formula \cite[Theorem
  7.4]{Iwaniec:2002a}, \cite{Selberg:1989a} and 
  integrate against $f$.
\end{proof}
\begin{remark}
  We note that the integral $\int_{\GmodH}f(z)K(z,z)d\mu(z)$ can be
  interpreted as the trace of the operator $M_fK$, where $M_f$ is
  multiplication by $f$. The operator  $M_fK$ has kernel $f(z)K(z,w)$,
  which is \emph{not} a point-pair invariant for $\psl$.
\end{remark}
\begin{remark}
  We remark that the first term on the right-hand side of Proposition
  \ref{local-trace-formula} is $\overline f$ times the contribution of the discrete spectrum 
  to the  Selberg trace formula \cite[Theorem
  10.2]{Iwaniec:2002a}, \cite{Selberg:1989a}. We shall see that for
  groups like $\pslz$ the two last terms can be estimated using Theorem \ref{average-best-qe} and the  Maa\ss{}--Selberg relations.
\end{remark}

\section{Approximation in the hyperbolic lattice point problem}\label{approx}

Let $\chi_A$ denote the characteristic function of a  set $A$. One would like  to use $k(u)=\chi_{[0, (X-2)/4]}(u)$ in the pre-trace formula, since the corresponding integral kernel is
exactly $N(z,w,X)$. Unfortunately the
decay in $t$ of the corresponding Selberg--Harish-Chandra transform $h(t)$ is
not strong enough to  analyze effectively the right-hand side of the
pre-trace formula. Therefore, it is better
to smooth $k(u)=\chi_{[0, (X-2)/4]}(u)$. Various types of smoothing
are appropriate depending on the problem at hand. 

For kernels $k_1$ and $k_2$ their \emph{hyperbolic convolution}  \cite{Chamizo:1996b} is
defined as 
\begin{equation*}k_1*k_2(u(z, w))= \int_{\H}k_{1}(u(z, v) )k_2(u(v, w))\, d\mu (v). \end{equation*}
The Selberg--Harish-Chandra transform of the
convolution is the pointwise product of the individual
Selberg--Harish-Chandra transforms,
i.e. $$h_{k_1*k_2}(t)=h_{k_1}(t)\cdot h_{k_2}(t),$$ see \cite[p.~323]{Chamizo:1996b}. In this paper we
will use the (non-smooth) mollifier
\begin{equation*}
  k_{\delta}(u)=\frac{1}{4\pi
  \sinh^2(\delta/2)}\chi_{[0, (\cosh\delta -1)/2]}(u).
\end{equation*}
with  \lq small\rq{} parameter $\delta$. This kernel satisfies
$\int_{\H}k_{\delta}(u(z,w))d\mu(z)=1$. The main reason for using this mollifier
rather than a smooth one is that we can compute its Selberg--Harish-Chandra transform
explicitly. Indeed for any indicator function $\chi_{[0,(\cosh R-1)/2]}(u)$
its transform equals
\begin{equation}\label{makeittick}
  h(t)=2^{5/2}\int_{0}^R(\cosh R-\cosh r)^{1/2}\cos (rt)dr=2\pi
  \sinh(R)P_{-1/2+it}^{-1}(\cosh R),
\end{equation}
where $P_{\mu}^{\nu}(z)$ is the associated Legendre function of the
  first kind.  Many properties of the kernels we
shall choose later follow from (\ref{makeittick}). Lemma 2.4 in \cite{Chamizo:1996b}
gives the following estimates:
\begin{equation}\label{trivialbound}
  h(t)=O((R+1)e^{R/2})
\end{equation}
uniformly for $t$ real. Furthermore
\begin{equation*}
h(t)=2|t|^{-3/2}\sqrt{2\pi \sinh R}\cos (Rt-(3\pi /4 )\hbox{sgn}t)+O(t^{-5/2}e^{R/2})
\end{equation*}
for $t$ real, $\abs{t}\geq 1$,  and $R\ge 1$. For $0\le R\le 1$ and
$t$ real and $\abs{t}\geq 1$ we have
\begin{equation*}
h(t)=2\pi Rt^{-1}J_1(Rt)\sqrt{\frac{\sinh R}{R}}+O(R^2 \min (R^2, |t|^{-2}),
\end{equation*} where $J_1$ is the Bessel function of order $1$.
Moreover, for every $R>0$ we have $h(i/2)=2\pi(\cosh R -1)$. The value $t=i/2$ corresponds to the eigenvalue
$\l=0$, which  gives the main term  for $\pslz$, as $\pslz$ has no
small eigenvalues. In general we can use  \cite[Lemma
2.4 (b)]{Chamizo:1996b}
 to analyze the contribution of small eigenvalues.

Given $X>0$ we define $R$ to be the positive solution of the equation  $1+2X=\cosh R$. We also define $Y$
through $\cosh Y=X/2$ with $Y>0$. With these definitions  $u\leq (X-2)/4$
precisely when $d\leq Y$. 

 Given $Z>0$,
using the triangle inequality for the hyperbolic distance,  $d(z,w)\leq
d(z,v)+d(v,w)$,  it is straightforward to verify that
\begin{equation}\label{from-triangle-inequality}
  \chi_{[0,(\cosh(Z)-1)/2]}*k_{\delta}(u(z,w))=\begin{cases}1,\textrm{ if
    }d(z,w)\leq Z-\delta,\\
0,\textrm{ if
    }d(z,w)\geq Z+\delta.
\end{cases}
\end{equation}
We now define functions with values in $[0,1]$ by
\begin{equation*}
k_{\pm}:=\chi_{[0,(\cosh(Y\pm \delta)-1)/2]}*k_{\delta}
\end{equation*}
and denote the corresponding Selberg--Harish-Chandra transforms by
$h_{\pm}$. 
Using \eqref{from-triangle-inequality} we now see that
\begin{equation}\label{ineq1}
  k_{-}(u)\leq \chi_{[0,(X-2)/4]}(u)\leq k_{+}(u).
\end{equation}
This inequality allows to  pass from  smoothed  kernels to the  sharp
cut-off.  In the following $X$ will be a large parameter tending to
infinity, and $\delta>0$ a small parameter tending to zero, given  as a
function of $X$. We notice that $\sinh (Y\pm \delta)=O(X)$. Using the
above general bounds for $h(t)$ we have for   $0<\delta<1$
\begin{equation*}
h_{\delta}(t)=\frac{1}{2}\frac{\sqrt{\delta\sinh(\delta)}}{\abs{t}\sinh^{2}{(\delta/2)}}J_1(\delta
\abs{t})+O(\delta^2\min(1,(\delta\abs{t})^{-2})),
\end{equation*}
where $h_{\delta}(t):=h_{k_{\delta}}(t)$.
Therefore, we have the following estimates for the
Selberg--Harish-Chandra transforms $h_{\pm}$ 
of
$k_{\pm}$ 
that are valid for $t$ real and $\abs{t}\geq 1$, 
and $Y-\delta>1$:
\begin{align}\nonumber
  h_{\pm}(t)&=H_{\pm}(t)\\
\label{uddannelse}&\quad+O\left(\frac{X^{1/2}}{\abs{t}^{3/2}}\left(\delta^2\min(1,(\delta\abs{t})^{-2})+\abs{t}^{-1}\min(1,(\delta\abs{t})^{-3/2})\right)\right),
\end{align}
where 
 \begin{equation}\label{justdoit}H_{\pm}(t)=\frac{\sqrt{2\pi \delta\sinh(\delta) \sinh (  Y    \pm \delta) }}{|t|^{5/2} \sinh^{2} (\delta/2)}J_1(\delta |t|)\cos ((      Y   \pm \delta)t-(3\pi /4) \hbox{sgn} t).\end{equation}
The error term is found by
multiplying the error term for $h_{\delta}$ by a bound for the main term
of the transform $h$ of $\chi_{[0,(\cosh(Y\pm \delta)-1)/2]}$ and then adding a
bound for the main term for $h_{\delta}$ with the error term for
$h$. 

Using Weyl's law and the above bounds, it is straightforward to verify that for 
$\delta$  bounded we have
\begin{equation}
  \label{focus-on-mainterms}
\sum_{\abs{t_j}\geq1}(h_{\pm}(t_j)-H_{\pm}(t_j))=O(X^{1/2}).
\end{equation}
We notice that by writing the cosine in $H_{\pm}(t)$ as a sum
of exponentials we can write
\begin{equation}\label{H-splitting}
  H_{\pm}(t)=A_{\pm}(t, X,\delta)e^{it(Y\pm\delta)}+B_{\pm}(t, X,\delta)e^{-it(Y\pm\delta)}.
\end{equation}
We use 
\begin{equation} \label{J1bound}
  J_1(z)=O(\min (\abs{z},\abs{z}^{-1/2})),\quad  J'_1(z)=O(\min (1,\abs{z}^{-1/2}),
\end{equation}
see \cite[Appendix B4]{Iwaniec:2002a}
to get  for $\abs{t}\geq 1$
\begin{equation}\label{functionbounds}
  A_{\pm}(t, X,\delta),
  B_{\pm}(t, X,\delta)=O(X^{1/2}\abs{t}^{-3/2}\min(1, (\delta\abs{t})^{-3/2}))
\end{equation}
and 
\begin{equation}\label{derivativebounds}
  A'_{\pm}(t,X,\delta),  B'_{\pm}(t, X,\delta)=O(X^{1/2}\abs{t}^{-5/2}\min(1, (\delta\abs{t})^{-1/2})).
\end{equation}

\begin{remark}
The smoothed functions $h_{\pm}(t)$  decay as
$O(\abs{t}^{-5/2})$, which is better than the rate of decay of the non-smooth ones i.e.
$O(\abs{t}^{-3/2})$. If we use $k_{\delta}*\cdots *k_{\delta}$ rather than just $k_{\delta}$,
the rate of decay becomes even better (an additional $\abs{t}^{-1}$ for each
extra convolution). Unfortunately this does not  improve  the
  final bound. Notice that $1*k_{\delta}\cdots *k_{\delta}=1$, i.e. $k_{\delta}*\cdots *k_\delta$
has integral $1$.
\end{remark}

\section{Upper bounds}\label{hedgehog}
In this section we will prove Theorem \ref{maintheorem}. To do so we will use
the kernels $k_{\pm}$ constructed in the previous section. We will assume
that $\delta=X^{-c}$ for some $c>0$.

We  analyzing the various terms of the spectral side in
Proposition \ref{local-trace-formula}. 

We start evaluating the contribution of  the continuous spectrum to the integrated
pre-trace formula.
\begin{lem}If $\G$ is a congruence group
\label{que-contrib-continuous}
  \begin{equation*}
 \int_\R
h_{\pm}(t)\left(\int_{\GmodH}f(z)\abs{E(z,1/2+it)}^2d\mu(z)
\right)dt =O( X^{1/2}\log X).    
  \end{equation*}
\end{lem}
\begin{proof}
  By using that $f$ has support in $$F_a=\{z\in \H, \abs{z}>1,
  \abs{\Re(z)}\leq 1/2, \Im (z)\leq a\}$$
for $a$ sufficiently large, we see that  
  \begin{align*}
  \int_{\GmodH}f(z)\abs{E(z,1/2+it)}^2d\mu(z)&\leq C
  \int_{F_a}\abs{E(z,1/2+it)}^2d\mu(z)\\
&\leq C \norm{E^a(z,1/2+it)}^2,
  \end{align*}
where $E^a(z,1/2+it)$ is the truncated Eisenstein series (see
\cite[(7.39)]{Selberg:1989a}). By the Maa\ss{}--Selberg relations we have
\begin{equation*}
  \norm{E^a(z,1/2+it)}^2=O_a\left(1+\abs{\frac{-\phi'}{\phi}(1/2+it)}\right),
\end{equation*}
 see \cite[(7.42')]{Selberg:1989a}.
For congruence groups the scattering matrix can be computed, and this
leads to 
\begin{equation}\label{scattering-estimate}\frac{-\phi'}{\phi}
      \left( \frac{1}{2}+it\right)=O(\log {(2+\abs{t})}),\end{equation}
      see e.g. \cite[Eq. 2.5
p.508]{Hejhal:1983a} for $\pslz$.    
It follows that 
\begin{equation}\label{maass-selberg-estimate}\int_{\GmodH}f(z)\abs{E(z,1/2+it)}^2d\mu(z) =O_f(\log {(2+\abs{t})}).\end{equation}
      Using
      \eqref{trivialbound}, \eqref{uddannelse}, and
      \eqref{functionbounds}, we get
      \begin{equation}\label{hpmestimate}
        h_\pm(t)=O(X^{1/2}\abs{t}^{-3/2})
      \end{equation}
for $\abs{t}\geq 1$. It follows that $$\int_{\abs{t}>1}
h_{\pm}(t)\left(\int_{\GmodH}f(z)\abs{E(z,1/2+it)}^2d\mu(z)
\right)dt =O( X^{1/2}).$$    For $\abs{t}<1$ we use $h_\pm(t)=O(X^{1/2}\log
X)$, see \eqref{trivialbound}. We deduce that
\begin{equation*}
\int_{\abs{t}\leq 1}
h_{\pm}(t)\left(\int_{\GmodH}f(z)\abs{E(z,1/2+it)}^2d\mu(z)
\right)dt =O( X^{1/2}\log X).
\end{equation*}
\end{proof}

\begin{lem}\label{good}
  \begin{equation*}
    \sum_{t_j }h_{\pm}(t_j)\left(\int_{\GmodH} f(z)d\mu_j(z)-\overline{ f} \right)=O(X^{1/2+\e}).
  \end{equation*}
\end{lem}
\begin{proof}
The eigenvalue $\lambda=0$ does not contribute since $$\int_{\GmodH} f(z)d\mu_0(z)-\overline{ f} =0.$$
Therefore, using \eqref{trivialbound}, we  only need to sum over $\abs{t_j}\geq
1$. This can be bounded by using Theorem 
\ref{average-best-qe}, the Cauchy--Schwarz inequality, Weyl's law, and dyadic
decomposition. We find
\begin{align*}
  \abs{\sum_{T<\abs{t_j}\leq 2T}h_{\pm}(t_j)\right.&\left.\left(\int_{\GmodH}
      f(z)d\mu_j(z)-\overline{ f}\ \right)}\\
& \leq \left(\sum_{T<\abs{t_j}\leq 2T}\abs{h_{\pm}(t_j)}^2\right)^{1/2}
\left(\sum_{T<\abs{t_j}\leq 2T}\abs{\int_{\GmodH}
    f(z)d\mu_j(z)-\overline{ f} }^2\right)^{1/2}\\
&=O_f( \max_{T\leq {t}\leq 2T}\abs{h_{\pm}(t)}T^{3/2+\varepsilon}).
\end{align*}
Using \begin{equation*}
  H_{\pm}(t)=O(\abs{t}^{-3/2}X^{1/2}\min(1, (\delta\abs{t}))^{-3/2}),
\end{equation*}
which follows from \eqref{functionbounds}, and
 \eqref{uddannelse} we find
 \begin{equation}
 O_f( \max_{T\leq {t}\leq
   2T}\abs{h_{\pm}(t)}T^{3/2+\varepsilon})=O_f(T^\varepsilon X^{1/2}\min(1,(\delta
 T)^{-3/2})).
\end{equation}
We now observe that the dyadic sums over $T=2^n$ satisfies
\begin{equation*}
  \sum_{n=0}^{\infty}2^{n\e}X^{1/2}\min(1,(\delta 2^n)^{-3/2}) =O(X^{1/2+\varepsilon}).
\end{equation*}
This can be  deduced by splitting the sum at $\delta2^n=1$ and computing the
resulting geometric series, and using $\delta=X^{-c}$ for some
$c>0$. The result follows.
\end{proof}
\begin{lem}\label{worse}For $\Gamma=\pslz$
the following estimate holds:
  \begin{equation}
  \sum_{t_j \in\mathbb R}h_{\pm}(t_j) +\frac{1}{4\pi}\int_\R h_\pm(t)\frac{-\phi'}{\phi}
      \left( \frac{1}{2}+it\right) dt=O(X^{7/12+\varepsilon}).
  \end{equation}
 \end{lem}
 \begin{proof}
   Using (\ref{scattering-estimate}), (\ref{trivialbound}), (\ref{hpmestimate})  we 
prove that  the integral is $O(X^{1/2+\e})$. This is similar to the last part of the proof of Lemma \ref{que-contrib-continuous}.  There are finitely many terms with $\abs{t_j}<2$ and each is  $O(X^{1/2}\log X)$ by
\eqref{trivialbound}. So we need to estimate the sum
\begin{equation*}
  \sum_{2\leq\abs{t_j}  } h_{\pm}(t_j).
\end{equation*}
By \eqref{focus-on-mainterms} we see that it
suffices to bound 
\begin{equation*}
\sum_{2\leq \abs{t_j} } H_\pm(t_j)  .
\end{equation*}
We now consider
\begin{equation}\label{beef}
\sum_{T< \abs{t_j}\leq 2T } H_\pm(t_j).
 \end{equation}
Using \eqref{H-splitting} this equals
\begin{equation*}
\sum_{T< \abs{t_j}\leq 2T } A_{\pm}(t_j, X,\delta)e^{it_j(Y\pm\delta)}
  \end{equation*}
plus a similar expression with $B$ instead of $A$. We recall the definition  \eqref{exp-sum} and use summation by
parts to  write the sum as 
\begin{equation*}
  A_\pm(2T,X,\delta)S(2T,e^{Y\pm\delta})-A_\pm(T,X,\delta)S(T,e^{Y\pm\delta})-\int_T^{2T}A'_\pm(u,
  X, \delta)S(u,e^{Y\pm\delta})du.
\end{equation*}
Therefore,  using \eqref{functionbounds},
\eqref{derivativebounds}, $e^{Y\pm \delta}=O(X)$, and \emph{any} bound of the form
\begin{equation*}
  S(T,X)=O(X^aT^b),
\end{equation*}
we find 
 that \eqref{beef} is estimated as
\begin{equation*}
  O(X^{1/2+a}T^{b-3/2}\min(1,(\delta T)^{-1/2})). 
\end{equation*}
By summing over dyadic intervals we get
\begin{align}
\label{whyohwhy}\nonumber\sum_{2\leq \abs{t_j} } H_\pm(t_j)=&O\left(X^{1/2+a}\sum_{n=1}^\infty
2^{n(b-3/2)}\min(1,(\delta 2^n)^{-1/2})\right)\\
=&O\left(X^{1/2+a}\left(\sum_{n\leq -\log_2( \delta)}
2^{n(b-3/2)}+\sum_{n> -\log_2( \delta)}2^{n(b-2)}\delta^{-1/2}\right)\right).
\end{align}
The first sum is bounded as long as $b<3/2$, while second sum is
$O(\delta^{3/2-b})$ if $b<2$. 
We therefore find that as long as
$b<3/2$ then \eqref{whyohwhy} is $O(X^{1/2+a})$. 

To get a good bound for $S(T,X)$ we interpolate  between the trivial bound
\eqref{trivial-bound} and the Luo--Sarnak bound
(Theorem \ref{luosarnak-bound}). To optimize we use the elementary inequality
\begin{equation}
  \label{elem-inequality}
  \min(k,l)\leq k^rl^{1-r}
\end{equation}
valid for $k,l>0$, $0\leq r\leq 1$.

 We  find that for any $0\leq r\leq 1$
\begin{equation*}
  S(T,X)=O((X^{\frac r 8}T^{r\frac 5 4+2(1-r)+\varepsilon}).
\end{equation*}
We let $a(r)=\frac r 8$ and $b(r)=r\frac 5 4+2(1-r)$.
For every $r>2/3$ we have $b(r)<3/2$. It follows that
for every $r>2/3$ the sum in \eqref{whyohwhy} is $O_r(X^{1/2+r/8+\varepsilon})$. The
result now follows.
\end{proof}
\begin{remark}\label{use-remark} We notice that if we knew \eqref{sufficient} then the
  above proof would give the optimal exponent $1/2+\e$ instead of
  $7/12+\e$. We note also that if we balance the error term from the zero eigenvalue,
  i.e. $O(X\delta)$ with the second error term $O(X^{1/2+a}\delta^{3/2-b})$ in \eqref{whyohwhy} with
  $(a,b)=(1/8,5/4+\varepsilon)$ from Theorem \ref{luosarnak-bound} we find
  $\delta=X^{-1/2+\e}$. This optimizes the error terms.\end{remark}

We can now prove Theorem \ref{maintheorem}: We observe that
\begin{equation} h_{\pm}(i/2)=2\pi(\cosh(Y\pm
\delta)-1)\frac{2\pi(\cosh(\delta)-1)}{4\pi\sinh^2(\delta/2)}=\pi X+O(X\delta).
\label{mainterm-contribution}\end{equation}
Using Lemmata \ref{que-contrib-continuous}, \ref{good}, and \ref{worse}, and Proposition \ref{local-trace-formula} we find 
  that 
$$\int_\GmodH f(z)K_\pm(z,z)=\overline f \pi X+O(X^{7/12+\varepsilon})+O(X\delta).$$ 
This gives by
\eqref{ineq1} 
\begin{equation*}
  \int_\GmodH f(z)\sum_{\gamma\in\G}1_{[0,(X-2)/4]}(u(\gamma z,z)))d \mu(z)=\overline f \pi X+O(X^{7/12+\varepsilon})+O(X\delta),
\end{equation*}  since $f$ is positive.
Theorem \ref{maintheorem} follows by choosing $\delta=X^{-1/2}$.

\section{Omega results}\label{hedgehogbabies}
In this section we investigate omega results for $N_f(X)$. 
\begin{theorem} \label{lowerbound-section} Let $f$ be a nonzero, nonnegative, compactly supported
   function on
  $\GmodH$. Then for every $\nu>0$ we have
  \begin{equation*}
    N_f(X)=\pi X \overline f+\Omega(X^{1/2}(\log\log x)^{1/4-\nu}).
  \end{equation*}
\end{theorem}
Since analogous omega results hold pointwise (see
\cite[Theorem 1.2]{PhillipsRudnick:1994a}) this is  not a surprising result. In
fact our proof below  is based on investigating the 
uniformity in $z$  in the proof in \cite[Theorem
1.2]{PhillipsRudnick:1994a}. The main ingredients in \cite[Theorem
1.2]{PhillipsRudnick:1994a}  are two lemmata. The
first assures that certain phases can be aligned, and the second
provides asymptotics (and in particular omega results) for an  \lq average local
Weyl law\rq. We  quote the alignment lemma  directly from \cite[Lemma
3.3]{PhillipsRudnick:1994a}. It can be proved using a simple
application of Dirichlet's box principle and the elementary inequality
$\abs{e^{i\theta}-1}<\abs{\theta}$ for $\theta\ne 0$:
\begin{lem} \label{box-principle}
 Given $n$ real numbers $r_1,\ldots,r_n$,  $M>0$, and $T>1$,  there exists an
 $s$ with $M\leq s \leq  MT^n$ such that
 \begin{equation*}
   \abs{e^{ir_js}-1}<\frac{1}{T}, \quad j=1,\ldots, n.
 \end{equation*}
\end{lem}
The \lq average local
Weyl law\rq{} is slightly more subtle, since the main term in the \lq
local Weyl law\rq{} proved in \cite[Lemma 2.3]{PhillipsRudnick:1994a}
depends on the point $z$:

\begin{lem}\label{local-average-weyl}
  Let $f$ be a smooth compactly supported function on $\GmodH$. Then 
  \begin{equation*}
    \sum_{|t_j|\leq
    T }\int_{\GmodH}f(z)d\mu_j(z)+\sum_{\mathfrak
    a}\frac{1}{4\pi}\int_{-T}^{T}\int_{\GmodH}f(z)\abs{E_{\mathfrak a}(z,1/2+it)}^2d\mu(z)\, dt\sim
    \frac{vol{(\GmodH)}}{4\pi}\overline f T^2
  \end{equation*}
as $T\to\infty$.
\end{lem}
\begin{proof}[Sketch of proof] This is a more or less standard
  application of the heat kernel and a Tauberian theorem. For
  $\delta>0$ we let $h(t)=e^{-\delta t^2}$. Using this as the spectral
  kernel in the
  pre-trace formula and integrating on the diagonal against $f$ we obtain
  \begin{align}
\nonumber   \sum_{\gamma\in \G}& \int_{\GmodH}f(z)k(u(\gamma z,z))d\mu (z)\\=
   \label{heat-expression}\sum_{t_j} &h(t_j)\int_{\GmodH}f(z)d\mu_j(z)+\sum_{\mathfrak
    a}\frac{1}{4\pi}\int_\R h(t)\int_{\GmodH}f(z)\abs{E_{\mathfrak a}(z,1/2+it)}^2d\mu(z)dt,
  \end{align}
  where $k$ is the inverse Selberg--Harish-Chandra transform of $h$. The
function $k(u)$ is decreasing in $u$ and satisfies
\begin{equation}\label{bound-on-geometric-heat-kernel}
  k\left(\frac{\cosh v -1}{2}\right)\leq \frac{C}
{\delta}e^{-\frac{v^2}{4\delta}},\end{equation} 
where $C$ is an absolute constant. This follows from an elementary
evaluation of $k$ (See \cite[Lemma 3.1]{Chamizo:1996a}). Furthermore
\begin{equation}
k(0)=\frac{1}{4\pi}\int_\R t \tanh(\pi t)h(t) dt=\frac{1}{4\pi \delta} +O(1)
\end{equation}
as $\delta \to 0$, which follows from  $\tanh(\pi
t)=1+O(e^{-2\pi{\abs{t}}})$. 
We now notice, since $u(\gamma z,z)=(\cosh(d(\gamma z,z))-1)/2$, that by \eqref{bound-on-geometric-heat-kernel}
\begin{equation}\label{all-nonzero-terms}
\sum_{I\neq\gamma\in \G}\int_{\GmodH}f(z)k(u(\gamma z,z))d\mu
(z)=O\left(\frac{1}{\delta}\int_{K\cap F} \sum_{I\neq \gamma\in \G }e^{-\frac{d(\gamma z,z)^2}{4\delta}}d\mu(z)\right),
\end{equation} where $K$ is the support of $f$ and $F$ is some 
Dirichlet fundamental domain.

We will show that the right-hand side in \eqref{all-nonzero-terms} is $o(1/\delta)$. This implies that the left-hand side of \eqref{heat-expression}
is asymptotic to ${vol{(\GmodH)}}\overline f/{(4\pi\delta)} $ as $\delta\to 0$. The claim of the
theorem now follows from Karamata's Tauberian theorem (see
\cite[Theorem 4.3]{Widder:1941a}).

To analyze
\begin{equation}\label{must-analyze}
  \frac{1}{\delta}\int_{K\cap F} \sum_{I\neq \gamma\in \G }e^{-\frac{d(\gamma z,z)^2}{4\delta}}d\mu(z)
\end{equation}
we split the sum as
\begin{equation}\label{splitting}
\sum_{I\neq \gamma }e^{-\frac{d(\gamma
    z,z)^2}{4\delta}}=\sum_{I\neq \gamma \in A }e^{-\frac{d(\gamma
      z,z)^2}{4\delta}}+\sum_{I\neq \gamma \in \Gamma\setminus A }e^{-\frac{d(\gamma z,z)^2}{4\delta}},
\end{equation}
where $A=\{\gamma \in \Gamma \vert d(\gamma w,w)\leq 1 \textrm{ for
  some } w\in K\cap F\}$. 

We consider the first sum. We claim that $A$ finite. To see this let
$$M=\{z\in \H, \exists w\in K\cap F \text{ with }d(z,w)\leq 1\}.$$ The set $M$ is 
 compact. Now note that  $B=\{\gamma\in \Gamma , M\cap
\gamma F\neq\emptyset \}$ contains $A$ and is finite by \cite[Theorem 1.6.2
(3)]{Miyake:2006a}. For $\gamma_0 \in A$ let $\epsilon>0$ and 
split $K\cap F $ as 
\begin{equation*}K\cap F=F_1(\epsilon)\cup F_2(\epsilon)=\{z\in K\cap F, d(\gamma_0 z,z)\leq \epsilon\}\cup \{z\in K\cap F, d(\gamma_0 z,z)> \epsilon\}.\end{equation*}
It is now clear that 
\begin{align*}
  \int_{K\cap F}e^{-\frac{d(\gamma_0
    z,z)^2}{4\delta}}d\mu(z) &=\int_{F_1(\epsilon)}e^{-\frac{d(\gamma_0
    z,z)^2}{4\delta}}d\mu(z)+\int_{F_2(\epsilon)}e^{-\frac{d(\gamma_0
    z,z)^2}{4\delta}}d\mu(z)
\\ &=O(\mu(F_1(\epsilon)))+O(e^{-\frac{\epsilon^2}{4\delta}}).
\end{align*}
By Lemma \ref{very-elementary-lemma} we have $\mu(F_1(\epsilon))=O(\epsilon^2)$.  Choosing
$\epsilon=\sqrt{4\delta}(-\log \delta)$ we see that 
\begin{equation*}
 \frac{1}{\delta}\int_{K\cap F} e^{-\frac{d(\gamma_0
      z,z)^2}{4\delta}}d\mu(z)=o(1/\delta).
\end{equation*}
 Since there are only finitely many terms in the sum
over $A$, this estimate suffices in dealing with this sum.

To handle the second sum in \eqref{splitting} we use that 
\begin{equation*}
\#\{\gamma \in \Gamma, d(\gamma z,z)\leq r \}=O(e^r),
\end{equation*}
where the implied constant is absolute for $z$ in a compact set and
and fixed  group $\G$ (see \cite[Corollary 2.12]{Iwaniec:2002a}).

It follows that for $z\in K$
\begin{align*}
  \sum_{I\neq \gamma \in \Gamma\setminus A }e^{-\frac{d(\gamma
      z,z)^2}{4\delta}}&=\sum_{n=0}^\infty \sum_{\substack{\gamma \in
      \Gamma \\2^n\leq d(\gamma
    z,z)\leq 2^{n+1}}}e^{-\frac{d(\gamma
      z,z)^2}{4\delta}}\\
&\leq \sum_{n=0}^\infty e^{-\frac{4^n}{4\delta}} O(e^{2^{n+1}})=O(e^{-\frac{C}{\delta}})
\end{align*}
for some absolute constant $C>0$. This suffices to conclude that 
\begin{equation*}
\frac{1}{\delta}\int_{K\cap F}  \sum_{I\neq \gamma \in \Gamma\setminus A
}e^{-\frac{d(\gamma z,z)^2}{4\delta}}d\mu (z)=o(1/\delta).
\end{equation*}
Collecting all the terms we find that  the left-hand side in \eqref{all-nonzero-terms} is
$o(\delta^{-1})$. This concludes the proof. \end{proof}

We have not been able to find a reference for the following elementary result:
\begin{lem}\label{very-elementary-lemma}
  Let $M\subseteq \H$ be any set. Let $\pm I\neq \gamma\in
  SL_2(\R)$. Then for sufficiently small $\epsilon$ we have
  \begin{equation*}
    \mu(\{z\in M , d(\gamma
    z,z)<\epsilon\})=\begin{cases}
0,&\textrm{ if $\gamma$ is hyperbolic,}\\
0,&\textrm{ if $\gamma$ is parabolic and $M$
        compact,}\\
O(\epsilon^2),&\textrm{ if $\gamma$ is elliptic.} 
\end{cases}
  \end{equation*}
\end{lem}
\begin{proof} Since $d(z,w)$ is a point-pair invariant we can assume
  that $\gamma$ is in canonical form. We can also assume
  $\epsilon<1$. If $d(\gamma z,z)<\epsilon$ then $u(\gamma
  z,z)=(\cosh d(\gamma z,z)-1)/2\leq \epsilon^2.$ 

Let $\gamma$ be hyperbolic. Then $\gamma z=pz$ for some
real number $p> 1$. Then
  \begin{equation*}
    u(\gamma z,z)=\frac{\abs{p-1}^2\abs{z}^2}{4py^2}\geq \frac{\abs{p-1}^2}{4p},
  \end{equation*}
which shows that $\{z\in M,  d(\gamma
    z,z)<\epsilon\}$ is empty for $\epsilon<{\abs{p-1}}/({2\sqrt{p}})$.

Let  $\gamma$  be parabolic and $M$ compact. Then  $\gamma
z=z+v$ for some $v\in\R\setminus\{0\}$, so that
  \begin{equation*}
    u(\gamma z,z)=\frac{\abs{v}^2}{4y^2}\geq C
  \end{equation*}
for some positive $C$. This shows that $\{z\in M,  d(\gamma
    z,z)<\epsilon\}$ is empty for $\epsilon<\sqrt{C}$.

Last let   $\gamma$ be elliptic, i.e. $\gamma z
=\frac{(\cos\theta) z+\sin \theta}{-(\sin\theta) z+ \cos \theta}$ for some
$\theta\not\in \pi \Z$. Then 
\begin{equation}\label{itsallthere}
  u(\gamma z,z)=\sin^2\theta\frac{\abs{1+z^2}^2}{4y^2}. 
\end{equation}
We claim that for some suitable constant $C>0$ depending only on
$\gamma$ we have
\begin{equation}\label{stupid-inclusion}
\{z, d(\gamma
    z,z)<\epsilon\}\subset \{z, \abs{\Re(z)}<C\epsilon\hbox{ and }\abs{\Im(z)-1}<C\epsilon\}
  .\end{equation}
As the last set has hyperbolic area $O(\epsilon^2)$, this completes the proof of the lemma. To prove the claim  \eqref{stupid-inclusion} we
note that, if $d(\gamma z,z)<\epsilon$, then  \eqref{itsallthere}  gives
\begin{equation*}
  \frac{(1-y^2+x^2)^2}{y^2}\leq
\frac{4\epsilon^2}{\sin^2\theta}\quad\textrm{ and }\quad \frac{4x^2y^2}{y^2}\leq
\frac{4\epsilon^2}{\sin^2\theta}. 
\end{equation*}
The second inequality shows that $\abs{\Re(z)}<C\epsilon.$ By the
first inequality we have 
\begin{equation*}
  -\frac{2\epsilon}{\abs{\sin\theta}}\leq \frac{(1-y^2+x^2)}{y}\leq \frac{2\epsilon}{\abs{\sin\theta}}.
\end{equation*} The upper inequality shows that $y$ is bounded away from
zero, and the lower inequality shows that $y$ is bounded. It follows
that for some positive constant $C'$, $-C'\epsilon\leq (1-y)(1+y)\leq
C'\epsilon$ and in turn $\abs{1-y}\leq C''\epsilon$.  This completes the proof of the claim \eqref{stupid-inclusion}. \end{proof}

\begin{remark}We note that the  result by
  Phillips--Rudnick \cite[Lemma 2.3]{PhillipsRudnick:1994a}  analogous to Lemma \ref{local-average-weyl} is \emph{not} uniform in $z$, as the main term depends on
  the size of the stabilizer of $z$. The  proof  above shows that, when we integrate, the contribution of elliptic points is small.   \end{remark}
We  now investigate the omega result for $N_f(X)-\pi \overline f X$. Recall that
$X/2=\cosh(Y)$. As we expect the order to be close to
$\sqrt{X}\sim\sqrt{2\sinh Y}$ we consider 
\begin{equation}\label{error} E(Y)=\frac{N_f(2\cosh(Y))-2\pi \overline f \cosh(Y)}{\sqrt{2\pi
    \sinh{Y}}}.\end{equation}
Similarly to the proof of the  upper bound it is convenient to smooth out
\eqref{error}. We use a smoothing technique similar to
Phillips--Rudnick. 

Consider a smooth, even function $\psi:\R\to\R$, satisfying that
$\psi,\hat\psi\geq 0$, $\int_{\R}\psi(x)dx=1$, $\hbox{supp}{(\psi)}\subseteq [-1,1]$.
For $0<\epsilon<1$ we let
$\psi_{\epsilon}(r)=\epsilon^{-1}\psi(\epsilon^{-1}r)$ which
approximates a $\delta$-distribution at $0$ as $\epsilon\to 0$.  We
 consider the smoothed-out function 
\begin{equation}\label{smooth-lower-bound}
  E_\epsilon(Y)=\int_{\R}\psi_\epsilon(R-Y)E(R)dR.
\end{equation}
By inserting the definition of $E(R)$ and interchanging sums we find
that 
\begin{align*}
  E_\epsilon(Y)=\int_{\GmodH}f(z)&\left(\sum_{\gamma\in\G
    }\int_{\R}\psi_{\epsilon}(R-Y)
    \frac{1}{\sqrt{2\pi\sinh{R}}}\chi_{[0,(\cosh R -1)/2]} (u(\gamma z,z))dR\right.\\
&-\left.\int_{\R}\psi_{\epsilon}(R-Y)\frac{2\pi \cosh(R)}{\sqrt{2\pi
      \sinh(R)}\vol{\GmodH}}dR\right)d\mu(z).
\end{align*}
We see that the infinite sum over $\G$ is an automorphic kernel
evaluated at the diagonal. The corresponding free kernel is 
\begin{equation*}
  k_{\epsilon,Y}(u)=\int_{\R}\psi_{\epsilon}(R-Y)
    \frac{1}{\sqrt{2\pi\sinh{R}}}\chi_{[0,(\cosh R -1)/2]}(u)dR,
\end{equation*}
 and its corresponding Selberg--Harish-Chandra
transform (compare \cite[(1.62')]{Iwaniec:2002a}) is 
\begin{align}\label{anothersmoothedkernel}
  h_{\epsilon,Y}(t)=4\pi\int_{0}^\infty F_s(u)k_{\epsilon,Y}(u)du=\int_{\R}\psi_{\epsilon}(R-Y)
    \frac{1}{\sqrt{2\pi\sinh{R}}}h_R(t)dR.
\end{align}
Here $s=1/2+it$, $F_s(u)$ is the Gauss hypergeometric function $F(s, 1-s;1, u)$, and
$h_R$ is the Selberg--Harish-Chandra transform of $\chi_{[0,(\cosh R
  -1)/2]}(u)$. We note that the smoothed-out   kernel has as Selberg--Harish-Chandra
transform  the smoothing of the initial transform. 

We notice also that
\begin{equation*}
h_{\epsilon,Y}(i/2)=\int_{\R}\psi_\epsilon(R-Y)\frac{2\pi(\cosh R-1)}{\sqrt{2\pi\sinh
  R}}dR.
\end{equation*}
Assuming that  $h_{\epsilon,Y}(t)$ decays sufficiently fast, which we will verify below, it follows from the pre-trace formula that
\begin{align}
\label{spectralrepresentation}  E_{\epsilon}(Y)=\int_{\GmodH}f(z)\Big(\sum_{t_j\neq
   i/2}&h_{\epsilon,Y}(t_j)\abs{u_j(z)}^2
  \\
&+\left.\frac{1}{4\pi}\int_\R
 h_{\epsilon,Y}(t)\abs{E(z,1/2+it)}^2dt\right)d\mu (z)+O(1). \nonumber
\end{align}
To get better control of $h_{\epsilon,Y}$ we   compute $h_R(t)$ explicitly in terms of
special functions (see \cite[(2.8), (2.9)]{Chamizo:1996b} and subsequent discussion). 
 We have \begin{equation}\label{willitend}
  h_R(t)=\sqrt{2\pi\sinh{R}}\cdot\Re\left(e^{itR}\frac{\Gamma(it)}{\Gamma(3/2+it)}F(-1/2,3/2,1-it,(1-e^{2R})^{-1})\right),
\end{equation}
where $F$ is the Gauss hypergeometric function. For $t$ real and nonzero
\begin{equation}\label{hypergeometric-bound}
F(-1/2,3/2,1-it,(1-e^{2R})^{-1})=1+O(\min(1,\abs{t}^{-1})e^{-2R})
\end{equation}
and Stirling's approximation \cite[(5.113)]{IwaniecKowalski:2004a}
gives for $\abs{t}\geq 1$
\begin{equation*}
  \frac{\Gamma(it)}{\Gamma(3/2+it)}=e^{-\frac{3\pi i}{4}\mathrm{sgn}
    t}\abs{t}^{-3/2}(1+O(\abs{t}^{-1})).
\end{equation*}
Inserting these expressions in \eqref{anothersmoothedkernel} we find   that
\begin{equation}\label{great-expression}
  h_{\epsilon,Y}(t)=\Re\left(e^{i(tY-\frac{3\pi}{4}\mathrm{sgn}t)} \frac{\widehat\psi_\e(t)}{\abs{t}^{3/2}}\right)+O(\abs{t}^{-5/2})
\end{equation}
for  $\abs{t}\geq1$.
Since $\widehat\psi_\e(t)=O_m((\epsilon
\abs{t})^{-m})$ for  $m\in \mathbb N$, we have  $h_{\epsilon,Y}(t)=O_\epsilon(\abs{t}^{-5/2}).$ 

We recall that for $\pslz$ the scattering function $\phi (s)$ has the special value $\phi (1/2)=-1$. This follows from the explicit calculation of $\phi (s)=\xi (2-2s)/\xi (2s)$, where $\xi(s)$ is the completed Riemann zeta function.  This implies, through the functional equation $E(z, s)=\phi (s)E(z, 1-s)$ that the Eisenstein series vanishes identically at $s=1/2$.   
\begin{lem} Let $\G=\pslz$. For positive $f$ as above
\label{lowerbound-contrib-continuous}
  \begin{equation*}
 \int_\R
h_{\varepsilon,Y}(t)\int_{\GmodH}f(z)\abs{E(z,1/2+it)}^2d\mu(z)
dt =O_f( 1).    
  \end{equation*}
 uniformly in $0<\epsilon<1$.
\end{lem}
\begin{proof}
We  consider
  \begin{equation*}
\varphi(t)=\int_{\GmodH}f(z)\abs{E(z,1/2+it)}^2\,d\mu(z). 
\end{equation*} 
We have $\varphi(0)=0$ and using the Maa{\ss}--Selberg relations we find that $\varphi(t)=O(\log(2+\abs{t}))$  for
$t\in \R$. This is where we use
crucially that $\G=\pslz$. 
 We have 
\begin{align}
\label{splitting-a}   \int_\R  h_{\varepsilon,Y}(t)\int_{\GmodH}f(z)\abs{E(z,1/2+it)}^2dt
= \int_\R
h_{\varepsilon,Y}(t)\varphi(t)dt .\end{align}
To bound \eqref{splitting-a}  we notice that by \eqref{willitend} and
\eqref{hypergeometric-bound} that 
\begin{equation*}
  \frac{h_R(t)}{\sqrt{2\pi \sinh(R)}}\varphi(t)=O_f\left(\abs{\frac{\Gamma(it)}{\Gamma(3/2+it)}\varphi(t)}\right).
\end{equation*}
 Since
$\abs{E(z,1/2+it)}^2=E(z,1/2+it)E(z,1/2-it)$ is meromorphic as a
function of $t\in \C$, and holomorphic for $t\in \R$ and $\varphi(0)=0$, the function $$\frac{\Gamma(it)}{\Gamma(3/2+it)}\varphi(t)$$ is
holomorphic for $t\in \R$, i.e. the pole of the $\G$-function in the numerator cancels with the zero of $\varphi (t)$. Moreover, we have the bound  
$$\frac{\Gamma(it)}{\Gamma(3/2+it)}\varphi(t)=O_f((1+\abs{t})^{-3/2}\log(2+\abs{t})).$$ It follows that 
\begin{equation*}
  \int_\R h_{\varepsilon,Y}(t)\varphi(t)dt=O_f\left(\int_{\R}\int_{\R}\psi_\epsilon(R-Y)\abs{\frac{\Gamma(it)}{\Gamma(3/2+it)}\varphi(t)}dRdt\right)=O_f(1)
\end{equation*}
  uniformly in $\epsilon.$
\end{proof}

\begin{lem}\label{smoothlowerbound} Let $\nu>0$. Then for
  every $k$ and $R>0$ there exist $\epsilon\in(0, 1)$, and $Y_0>R$ such
  that 
  \begin{equation*}
    -E_{\epsilon}(Y_0)>k(\log(Y_0))^{1/4-\nu}. 
  \end{equation*}
\end{lem}
\begin{proof} It follows from Lemma
  \ref{lowerbound-contrib-continuous}, \eqref{spectralrepresentation} that 

\begin{equation*}
  E_\epsilon(Y)=\sum_{1\leq\abs{t_j}}h_{\epsilon,Y}(t_j)\int_{\GmodH}f(z)\abs{u_j(z)}^2d\mu(z)+O_f(1),
\end{equation*}
where we have used that for $\G=\pslz$ there are no non-zero
eigenvalues with $\abs{t_j}<1. $ Using \eqref{great-expression},
$\int_{\GmodH}f\abs{u_j}^2d\mu\leq \norm{f}_\infty$, and Weyl's law this equals
\begin{equation}
  \sum_{1\leq \abs{t_j}}\Re\left(e^{i(t_jY-\frac{3\pi}{4}\mathrm{sgn}t_j)} \frac{\widehat\psi_\e(t_j)}{\abs{t_j}^{3/2}}\right)\int_{\GmodH}fd\mu_j(z)+O_f(1).
\end{equation}
We split the sum at $T$, and bound the partial sum and the tail
separately. The precise value of $T$ will be chosen later. We use $\int_{\GmodH} fd\mu_j(z)\leq \norm{f}_\infty$ and
$\widehat\psi_\e(t)=O_m((\epsilon \abs{t})^{-m})$ for any $m$ to bound the tail as follows:
\begin{eqnarray*}
  \sum_{\abs{t_j}>T}\Re\left(e^{i(t_jY-\frac{3\pi}{4}\mathrm{sgn}t_j)}
    \frac{\widehat\psi_\e(t_j)}{\abs{t_j}^{3/2}}\right)\int_{\GmodH}fd\mu_j(z)&=&O_{f,m}(\epsilon^{-m}\sum_{\abs{t_j}>T}\abs{t_j}^{-(3/2+m)})\\&=&O_{f,m}(\epsilon^{-m}T^{1/2-m}),
\end{eqnarray*}
where we have used Weyl's law to estimate the sum of the series.

To analyze the partial sum Lemma \ref{box-principle} allows us to
choose $Y=Y_0>0$ depending on $T$, $R$, $V$  with 
$R\leq Y_0\leq RV^{N(T)-1}$ such that $\abs{e^{it_jY_0}-1}<V^{-1}$ for
\emph{all} $1<\abs{t_j}\leq T$. Here $N(T)-1$ is the number of
such $t_j$'s. Without loss of generality we can assume that
$t_j>0$.  Using the addition formula for cosine we see that 
\begin{equation*}
\abs{\cos(Y_0t_j-3\pi/4)-(-\sqrt{2}/2)}\leq V^{-1}.
\end{equation*}
Using Weyl's law we deduce that
\begin{align*}
   \sum_{1\leq t_j\leq
     T}&\Re\left(e^{i(t_jY_0-\frac{3\pi}{4}\mathrm{sgn}t_j)}
    \frac{\widehat\psi_\e(t_j)}{\abs{t_j}^{3/2}}\right)\int_{\GmodH}fd\mu_j(z)\\
&=-\frac{\sqrt{2}}{2}\sum_{1\leq
   t_j\leq T}\frac{\widehat\psi_\e(t_j)}{|t_j|^{3/2}}\int_{\GmodH}fd\mu_j(z)+O_f(T^{1/2}/V).
\end{align*}
 Since $\hat\psi(0)=1$ we
have, by an 
appropriate choice of $0<\tau<1$, that $\hat\psi(t)\geq 1/2$ for
$\abs{t}\leq \tau$. It follows that when $\abs{t}\leq \tau/\epsilon$
we have $\widehat\psi_\epsilon(t)\geq 1/2$. We note that all terms in the sum above are non-negative. Therefore, for $\tau/\epsilon<T$, we have
\begin{align*}
\frac{-\sqrt{2}}{2}\sum_{1\leq
   t_j\leq T}\frac{\widehat\psi_\e(t_j)}{|t_j|^{3/2}}\int_{\GmodH}fd\mu_j(z)&\leq \frac{-\sqrt{2}}{2}\sum_{1\leq
   t_j\leq \tau/\epsilon}\frac{\widehat\psi_\e(t_j)}{|t_j|^{3/2}}\int_{\GmodH}fd\mu_j(z)\\
&\leq -C\epsilon^{-1/2}.
\end{align*}
Here we have used Lemma \ref{local-average-weyl}, and $C$ is some
strictly positive constant depending only on $f$, $\tau$, $m$ and
$\G$. We note that in Lemma \ref{local-average-weyl} the contribution
from the continuous spectrum is $O(T\log T)$, as follows from \eqref{maass-selberg-estimate}.

To summarize we have proved that there exist $Y_0$, $R\leq Y_0\leq RV^{N(T)-1}$ with 
\begin{equation}
  \label{summarized}
-E_{\epsilon}(Y_0)\geq   C\epsilon^{-1/2}+O_{m,f}((1+\epsilon^{-m}T^{1/2-m}+T^{1/2}/V). 
\end{equation}
Let $\epsilon$, $T$ be chosen such that
$T^{1/2}/V=\epsilon^{-m}T^{1/2-m}=1$. If we choose $V$ such that $V>R$ (note that this puts an upper  bound on
$\epsilon$) and if we assume that
$R$ from the beginning was sufficiently large we have 
\begin{equation}
-E_{\epsilon}(Y_0)\geq   \frac{C}{2}V^{1-1/(2m)}.\label{eq:4}
\end{equation}
Since $V>R$ we have $Y_0\leq V^{N(T)}\leq V^{c_1 T^2}=
V^{c_1V^4}$ which forces $c_2V^4\geq \log(Y_0)
/\log\log(Y_0)$.
Now we choose $m$ to satisfy
$(8m)^{-1}<\nu/2$,  so that  
\begin{equation*}
  V^{1-1/(2m)}\geq (c_2^{-1}\log(Y_0)
/\log\log(Y_0))^{\frac{1}{4}(1-1/(2m))}\geq c_3\log(Y_0)^{\nu/2}(\log Y_0)^{1/4-\nu}.
\end{equation*}
The proof is complete once we observe that we could assume  that $R$  had been chosen such that
$c_3\log(Y_0)^{\nu/2}>k$ for $Y_0\geq R$.   
\end{proof} 

\begin{proof}[Proof of theorem \ref{lowerbound-section}] Proof by contradiction. Assume  that for some $\nu>0$ $$\frac{N_f(X)-\pi \overline f
    X}{X^{1/2}}=O((\log\log X)^{1/4-\nu}).$$ It follows from
  \eqref{error} and $X=2\cosh Y$ that
  \begin{equation*}
    E(Y)=O((\log Y)^{1/4-\nu}),
  \end{equation*}
and, therefore,
\begin{equation*}
\abs{E_\e(Y)}\leq \int_\R \psi_\epsilon(R-Y)\abs{E(R)}dR\leq K  (\log
(Y+\epsilon))^{1/4-\nu}=O((\log Y)^{1/4-\nu})
\end{equation*}
uniformly for all $\epsilon$. But this contradicts Lemma \ref{smoothlowerbound}.
\end{proof}
\bibliographystyle{abbrv}
\def\cprime{$'$}

\addresseshere
\newpage
\pagestyle{empty}


\section*{Supplementary material}

# APPENDIX: NUMERICAL INVESTIGATION OF EXPONENTIAL SUMS OVER EIGENVALUES IN $\mathrm{PSL}_2(\mathbb{Z})\backslash\mathbb{H}$

NIKO LAAKSONEN

Let $\lambda_j = \frac{1}{4} + t_j^2$ be the eigenvalues of $\Delta$ in $\Gamma\backslash\mathbb{H}$, where $\Gamma = \mathrm{PSL}_2(\mathbb{Z})$. For $X > 1$, we define the following sum

$$S(T, X) = \sum_{|t_j| \leq T} X^{it_j},$$

which is symmetrised by including both $t_j$ and $-t_j$. Petridis and Risager [7, Conjecture 2.2] conjecture that up to a factor of the order of $X^\epsilon$, the sum has square root cancellation in $T$.

**Conjecture.** *For every $\epsilon > 0$ and $X > 1$ we have*

$$S(T, X) \ll_\epsilon T^{1+\epsilon} X^\epsilon.$$

We report on the numerical investigation of the function $S(T, X)$ and prove a theorem about its behaviour as $T \to \infty$ and $X > 1$ is fixed. Our investigation resulted in the following observations.

**Experimental Observation 1.** *The growth of $S(T, X)$ is consistent with the conjecture.*

**Experimental Observation 2.** *For a fixed $X > 1$, $S(T, X)$ has a peak of order $T$ whenever $X$ is equal to a power of a norm of a primitive hyperbolic class of $\Gamma$ or an even power of a prime number $p \in \mathbb{N}$.*

Experimental Observation 2 is also in agreement with the results of Chazarain [2] that for the wave kernel the singularities occur at the lengths of closed geodesics (or in our case when $\log X$ is a multiple of a length of a prime geodesic). The peaks at even powers of rational primes are due to the scattering determinant $\varphi$. Experimental Observation 2 leads us to prove asymptotics for $S(T, X)$ for a fixed $X > 1$. Let $\Lambda(X)$ be the von Mangoldt function extended to $\mathbb{R}$ by defining it to be 0 when $X$ is not equal to a power of a prime number. We also define a similar function $\Lambda_\Gamma$ for the norms of hyperbolic classes of $\mathrm{PSL}_2(\mathbb{Z})$ given by

$$\Lambda_\Gamma(X) = \begin{cases} \log(N(\mathfrak{p})), & \text{if } X = N(\mathfrak{p})^\ell,\, \ell \in \mathbb{N}, \\ 0, & \text{otherwise.} \end{cases}$$

Let $|F|$ be the volume of the fundamental domain of $\Gamma\backslash\mathbb{H}$. We prove the following theorem.

---

1

**Theorem 1.** *For a fixed $X > 1$, we have*

$$S(T,X) = \frac{|F|}{\pi}\frac{\sin(T\log X)}{\log X}T + \frac{T}{\pi}(X^{1/2} - X^{-1/2})^{-1}\Lambda_\Gamma(X)$$
$$+ \frac{2T}{\pi}X^{-1/2}\Lambda(X^{1/2}) + O\left(T/\log T\right),$$

*as $T \to \infty$.*

*Proof.* Let $\psi$ be a positive even test function supported on $[-1, 1]$ with $\int \psi = 1$. Then define $\psi_\epsilon(x) = \epsilon^{-1}\psi(x/\epsilon)$. So $\psi_\epsilon$ is supported on $[-\epsilon, \epsilon]$ and $\int \psi_\epsilon = 1$. Also, let $G$ be the convolution $G(r) = (\mathbb{1}_{[-T,T]} * \psi_\epsilon)(r)$ for some $\epsilon > 0$ to be chosen later. Define a function $h$, depending on $T$, $X$ and $\epsilon$, given by $h(r) = G(r)(X^{ir} + X^{-ir})$. Let $g$ be the Fourier transform of $h$ and denote the determinant of the scattering matrix $\Phi$ by $\varphi$. Then the Selberg Trace Formula [3, Theorem 10.2] says that

$$\mathscr{S}(T,X) + \frac{1}{4\pi}\int_{-\infty}^{\infty} h(r)\frac{-\varphi'}{\varphi}\left(\frac{1}{2} + ir\right)dr = I(T,X) + H(T,X) + E(T,X) + L(T,X),$$

where

$$\mathscr{S}(T,X) = \sum_{t_j > 0} h(t_j),$$

$$I(T,X) = \frac{|F|}{4\pi}\int_{-\infty}^{\infty} h(r)r\tanh(\pi r)dr,$$

$$H(T,X) = \sum_{\mathfrak{p}}\sum_{\ell=1}^{\infty}\left(N(\mathfrak{p})^{\ell/2} - N(\mathfrak{p})^{-\ell/2}\right)^{-1} g(\ell \log N(\mathfrak{p}))\log N(\mathfrak{p}),$$

$$E(T,X) = \sum_{\mathcal{R}}\sum_{0 < \ell < m}\left(2m\sin\frac{\pi\ell}{m}\right)^{-1}\int_{-\infty}^{\infty} h(r)\frac{\cosh\pi\left(1 - \frac{2\ell}{m}\right)r}{\cosh\pi r}dr,$$

$$L(T,X) = \frac{h(0)}{4}\operatorname{Tr}(I - \Phi(\tfrac{1}{2})) - h_\Gamma g(0)\log 2 - \frac{h_\Gamma}{2\pi}\int_{-\infty}^{\infty} h(r)\frac{\Gamma'(1 + ir)}{\Gamma(1 + ir)}dr,$$

where $\mathfrak{p}$ and $\mathcal{R}$ range over the primitive hyperbolic and elliptic classes of $\operatorname{PSL}_2(\mathbb{Z})$, respectively, and $h_\Gamma$ is the number of cusps of $\Gamma$. First observe that $S(T,X) = \mathscr{S}(T,X) + O(T\epsilon)$, so we can work with $\mathscr{S}$. For the identity motion we have

$$I(T,X) = \frac{|F|}{2\pi}\int_{-\infty}^{\infty} G(r)\cos(r\log X)r\tanh\pi r\, dr$$
$$= \frac{|F|}{\pi}\int_0^{\infty} G(r)\cos(r\log X)r\left(1 - \frac{2}{e^{2\pi r} + 1}\right)dr$$
$$= \frac{|F|}{\pi}(I_1(T,X) + I_2(T,X)).$$

From $I_1$ we obtain a part of the main term:
$$I_1(T,X) = \int_0^\infty G(r)\cos(r\log X)r\,dr$$
$$= \left(\int_0^{T-\epsilon} + \int_{T-\epsilon}^{T+\epsilon}\right) G(r)\cos(r\log X)r\,dr$$
$$= I_{11} + I_{12},$$

since $G$ is even and supported on $[-T-\epsilon, T+\epsilon]$. Then
$$I_{11} = \int_0^{T-\epsilon}\cos(r\log X)r\,dr = \frac{\sin((T-\epsilon)\log X)}{\log X}(T-\epsilon) + O(1),$$
$$I_{12} \ll \int_{T-\epsilon}^{T+\epsilon} r\,dr = O(T\epsilon).$$

Also,
$$I_2(T,X) = -\int_0^\infty G(r)\cos(r\log X)r\frac{2}{e^{2\pi r}+1}\,dr \ll \int_0^\infty r e^{-2\pi r}\,dr = O(1).$$

For $g(r)$ we compute
$$g(r) = \frac{1}{2\pi}\int_{-\infty}^\infty e^{-irt}h(t)\,dt$$
$$= \frac{1}{2\pi}\int_{-\infty}^\infty G(t)e^{-irt}(e^{it\log X} + e^{-it\log X})\,dt$$
$$= \frac{1}{2\pi}\left(\widehat{G}\left(\frac{r-\log X}{2\pi}\right) + \widehat{G}\left(\frac{r+\log X}{2\pi}\right)\right).$$

So in particular $g(\ell\log N(\mathfrak{p})) \sim T/\pi$ if $X = N(\mathfrak{p})^\ell$ and decays as $O((\ell\log N(\mathfrak{p}))^{-k-1}\epsilon^{-k})$ otherwise, for any $k \in \mathbb{N}$. For the elliptic terms we need to evaluate
$$\int_{-\infty}^\infty h(r)\frac{\cosh\pi(1-\frac{2\ell}{m})r}{\cosh\pi r}\,dr \ll \int_0^\infty \frac{e^{-2\pi r\ell/m} + e^{-2\pi r}}{1+e^{-2\pi r}}\,dr = O(1).$$

Hence $E(T,X)$ is bounded. By the explicit formula of $\varphi'/\varphi$ for $\mathrm{PSL}_2(\mathbb{Z})$, [3, 3.24], we have
$$\int_{-\infty}^\infty h(r)\frac{-\varphi'}{\varphi}\left(\frac{1}{2}+ir\right)dr = \int_{-\infty}^\infty h(r)\left(-2\log\pi + \frac{\Gamma'(\frac{1}{2}\pm ir)}{\Gamma(\frac{1}{2}\pm ir)} + 2\frac{\zeta'(1\pm 2ir)}{\zeta(1\pm 2ir)}\right)dr.$$
$$= C_1 + C_2 + C_3.$$

The integral $C_1$ is the Fourier transform of $G$ and is thus bounded. For $C_2$ we use Stirling asymptotics to get
$$C_2 = \int_{-\infty}^\infty h(r)\log(\frac{1}{4}+r^2)\,dr + O(1).$$

This is $O(\log T)$. The same computation shows that $L(T,X) = O(\log T)$. The remaining part of the main term comes from $C_3$. We first expand $h$ and isolate the important terms:
$$C_3 = 2\left(\int_{-T-\epsilon}^{-T+\epsilon} + \int_{-T+\epsilon}^{T-\epsilon} + \int_{T-\epsilon}^{T+\epsilon}\right)(X^{ir} + X^{-ir})G(r)\frac{\zeta'(1\pm 2ir)}{\zeta(1\pm 2ir)}\,dr.$$

The first and third integrals are bounded by $O(\epsilon \log T)$. Notice that $G(r) = 1$ in the range of the second integral, hence we can write it as

$$2 \int_{\frac{1}{2}-(T-\epsilon)i}^{\frac{1}{2}+(T-\epsilon)i} (X^{s-1/2} + X^{1/2-s}) \left( \frac{\zeta'(2s)}{\zeta(2s)} + \frac{\zeta'(2-2s)}{\zeta(2-2s)} \right) ds.$$

We separate this into two integrals by adding and subtracting the singular part:

$$C_3 = 2 \int_{\frac{1}{2}-(T-\epsilon)i}^{\frac{1}{2}+(T-\epsilon)i} (X^{s-1/2} + X^{1/2-s}) \left( \frac{\zeta'(2s)}{\zeta(2s)} - \frac{1}{2s-1} \right) ds$$
$$+ 2 \int_{\frac{1}{2}-(T-\epsilon)i}^{\frac{1}{2}+(T-\epsilon)i} (X^{s-1/2} + X^{1/2-s}) \left( \frac{\zeta'(2-2s)}{\zeta(2-2s)} - \frac{1}{(2-2s)-1} \right) ds$$
$$= 2(C_{31} + C_{32}).$$

For the first integral we move the contour to $\Re s = 1$ and for the second one to $\Re s = 0$. It is easy to see that the top and bottom parts of the contours yield $O(\log T)$. For the line at $\Re s = 1$ we get

$$C_{31} = \int_{1-(T-\epsilon)i}^{1+(T-\epsilon)i} (X^{s-1/2} + X^{1/2-s}) \left( \frac{\zeta'(2s)}{\zeta(2s)} - \frac{1}{2s-1} \right) ds$$
$$= \int_{-T+\epsilon}^{T-\epsilon} (X^{1/2+ir} + X^{-1/2-ir}) \left( \frac{\zeta'(2+2ir)}{\zeta(2+2ir)} - \frac{1}{1+2ir} \right) dr,$$

For the rest of the proof we will follow an argument similar to [5, Hilfssatz 2]. We start by writing out the Dirichlet series:

$$C_{31} = -\int_{-T+\epsilon}^{T-\epsilon} X^{1/2+ir} \sum_{n=1}^{\infty} \frac{\Lambda(n)}{n^{2+2ir}} dr + O(\log T)$$
$$= -\sum_{n \neq \sqrt{X}} \frac{\sqrt{X}\Lambda(n)}{n^2} \int_{-T+\epsilon}^{T-\epsilon} \left( \frac{X}{n^2} \right)^{ir} dr - X^{-1/2}\Lambda(X^{1/2}) \int_{-T+\epsilon}^{T-\epsilon} dr + O(\log T).$$

Since $X > 1$, the term in $C_{31}$ with the negative exponent gets absorbed into the error term. Hence,

$$\left| \int_{-T+\epsilon}^{T-\epsilon} X^{1/2+ir} \frac{\zeta'(2+2ir)}{\zeta(2+2ir)} dr + 2X^{-1/2}\Lambda(X^{1/2})(T-\epsilon) \right|$$
$$\leq \sum_{n \neq \sqrt{X}} \frac{\sqrt{X}\Lambda(n)}{n^2} \left| \frac{(\frac{X}{n^2})^{i(T-\epsilon)} - (\frac{X}{n^2})^{-i(T-\epsilon)}}{\log \frac{X}{n^2}} \right|$$
$$\ll 2\sqrt{X} \left| \frac{\zeta'(2)}{\zeta(2)} \right|.$$

So we see that $C_{31} = -2X^{-1/2}\Lambda(X^{1/2})(T-\epsilon) + O(\log T)$. A similar argument shows that $C_{32}$ has the same asymptotics. Letting $\epsilon = 1/\log T$ concludes the proof. □

We will now present plots of $S(T, X)$ in terms of both $T$ and $X$. In Figures 1 to 3 we have fixed $T = 800$ with $X \to \infty$, while in Figure 4 we are

NUMERICS OF SPECTRAL EXPONENTIAL SUMS                    5fixing $X$ with $T \to \infty$. Taking into account the conjecture, we plot the normalised sum $\Sigma(T, X) = S(T, X)T^{-1}$. In Figure 4 we present a comparison for different powers of $T$, which suggests that $1 + \epsilon$ is the correct exponent. The programs used for the plots are available on the website [4]. We used 53 000 eigenvalues from the data of Then [8] with 13 decimal digit precision. We have also used the data of Booker and Strömbergsson related to [1], which has a much higher precision of 53 decimal digits for 2 280 eigenvalues. We verify that the computations are robust, that is, the number of eigenvalues or their precision has no significant impact on our calculations. More details are available on the website [4].

Recall that we expect a peak of order $T$ at all even prime powers as well as powers of the norms of the primitive hyperbolic classes. The first few norms (up to 8 decimals) are given by

$$g_1 = 6.85410196 \qquad g_5 = 46.97871376$$
$$g_2 = 13.92820323 \qquad g_6 = 61.98386677$$
$$g_3 = 22.95643924 \qquad g_7 = 78.98733975$$
$$g_4 = 33.97056274 \qquad g_8 = 97.98979486.$$

These can be computed by expressing the norm in terms of the trace (see e.g. [3, pg. 146]).

We start by considering $\Sigma(T, X)$ in terms of $X$.

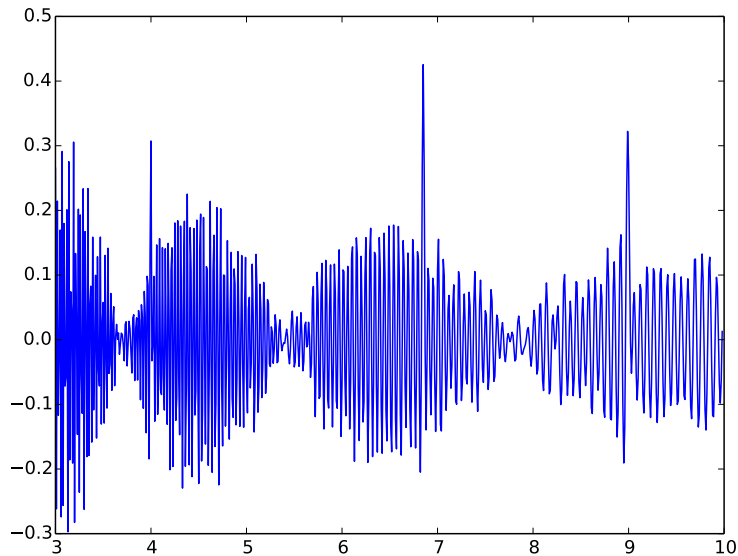

FIGURE 1. $\Sigma(T, X)$ in terms of $X$ for $X \in [3, 10]$.

This clearly shows peak points at $X = 4 = 2^2$, $X = g_1$ and $X = 9 = 3^2$. In the following plot we can see the peak points $X = g_2$ and $X = 16 = 2^4$:

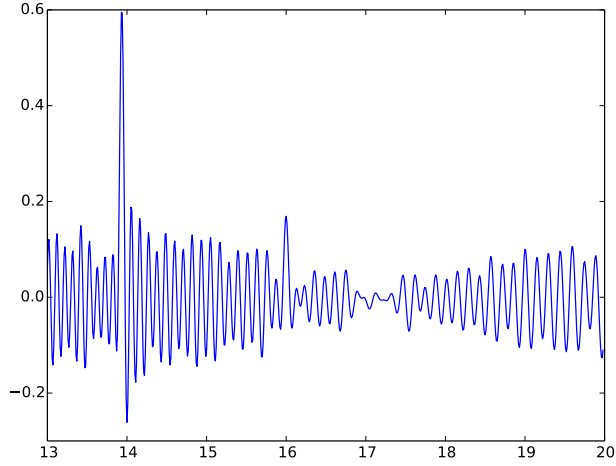

FIGURE 2. $\Sigma(T, X)$ in terms of $X$ for $X \in [13, 20]$.

Figures 1 and 2 verify Theorem 1 numerically in accordance with Experimental Observation 2. In Figure 3 we look at $\Sigma(T, X)$ for $X$ in a much larger interval. The graph agrees with Experimental Observation 1. On the other hand we cannot dispose of $X^\epsilon$ in the conjecture. The frequencies $t_j$ are conjecturally linearly independent over $\mathbb{Q}$, which makes $S(T, X)$ the partial sums of an almost periodic function. Therefore, for a choice of arbitrarily large $X$, compared to $T$, $S(T, X)$ will be of size $T^2$.

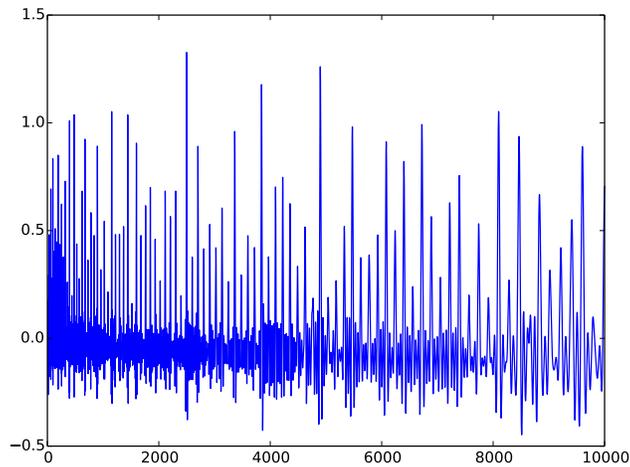

FIGURE 3. $\Sigma(T, X)$ in terms of $X$ for $X \in [100, 10\,000]$.

In Theorem 1 the asymptotics show an oscillatory term with an amplitude of order $T$ coming from the identity motion. We subtract it from $S(T, X)$

and define
$$\tilde{\Sigma}(T, X) = S(T, X) - \frac{|F|}{\pi}\frac{\sin(T \log X)}{\log X}T.$$

We plot $\tilde{\Sigma}(T, X)$ in terms of $T$ at $X = 49$, which is one of the peak points.

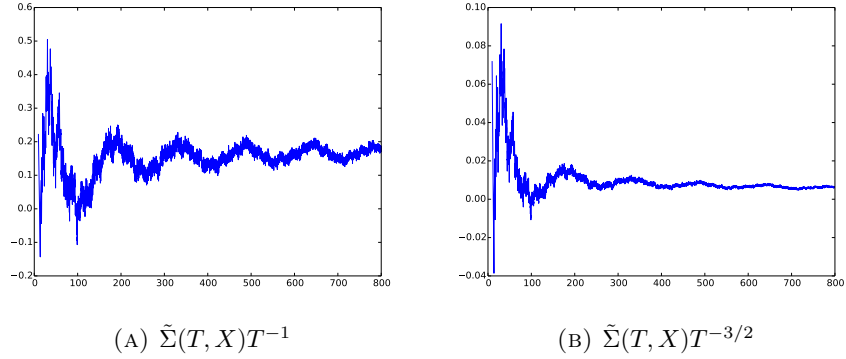

(A) $\tilde{\Sigma}(T, X)T^{-1}$

(B) $\tilde{\Sigma}(T, X)T^{-3/2}$

FIGURE 4. Different normalisations of $\tilde{\Sigma}(T, X)$ at $X = 49$.

Notice that clearly the normalisation $T^{-1}$ seems to be closer to the correct one, which is evidence towards our Experimental Observation 1.

It is of interest to compare the behaviour of $S(T, X)$ with a similar sum over the Riemann zeros. Landau [5, Satz 1] showed that for a fixed $x > 1$, if $\rho = \beta + i\gamma$ is a non-trivial zero of $\zeta(s)$, we have the formula

$$\sum_{0 < \gamma < T} x^\rho = -\frac{T}{2\pi}\Lambda(x) + O(\log T). \qquad (1)$$

We call the left-hand side of (1) $Z(T, X)$. We used our program with 10 000 zeros of $\zeta(s)$ to 9 decimal places, provided by Odlyzko [6]. With our program we obtain the following plot for the normalized sum $T^{-1}Z(T, X)$. Here blue denotes the real part and green the imaginary part of the sum.

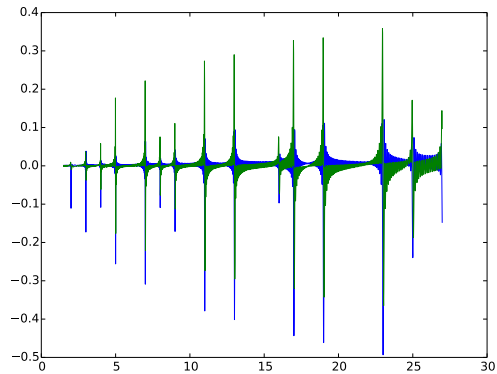

FIGURE 5. $x \in [1.5, 30]$

Notice that $S(T, X)$ is the analogue of the real part of $Z(T, X)$ only. Since the Selberg Trace Formula demands that the test function is even, we cannot analyse the imaginary part directly. For numerical study of this we again refer the reader to the website [4].

Department of Mathematics, University College London, Gower Street, London WC1E 6BT, United Kingdom

*E-mail address*: `n.laaksonen@ucl.ac.uk`


\end{document}